\documentclass{svmult}
\usepackage{graphicx,stmaryrd} 
\usepackage{amsfonts,amsmath,amssymb,tikz,hyperref}
\usepackage{biblatex}
\usepackage{comment}
\begin{document}
\title*{Modal Semantics for Reasoning with Probability and Uncertainty}
\author{Nino Guallart}

\maketitle

\begin{abstract}  \ 
This paper belongs to the field of probabilistic modal logic, focusing on a comparative analysis of two distinct semantics: one rooted in Kripke semantics and the other in neighbourhood semantics. The primary distinction lies in the following: The latter allows us to adequately express belief functions (lower probabilities) over propositions, whereas the former does not. Thus, neighbourhood semantics is more expressive.
The main part of the work is a section in which we study the modal equivalence between probabilistic Kripke models and a subclass of belief neighbourhood models, namely additive ones. We study how to obtain modally equivalent structures.


\end{abstract}

\section{Introduction}

The convergence of probability theory and modal logic has garnered significant attention across multiple disciplines.  Probabilistic modal logic has been the object of intense research in these fields over the last few decades.  Authors such as Nilsson \cite{nilsson1986probabilistic}, Bacchus \cite{bacchus1991representing}, Fagin and Halpern \cite{halpern1992two,halpern2003reasoning,fagin2013new,fagin1988reasoning,fagin1991uncertainty}, Aumann \cite{aumann1995backward,aumann1999interactive}, Samet, Heifetz and Mongin \cite{heifetz1998modal,heifetz2001probability,heifetz1998topology}
have applied modal logic for the formalisation of subjective probability in different fields: economics, artificial intelligence or philosophy of science.  The basic idea is to develop a probability space within some kind of structure, with Kripke models being the most obvious choice. Modal semantics allows the combination of probability with other modal operators, and hence its applicability to different kinds of modal logics such as epistemic logic or dynamic epistemic logic \cite{kooi2003probabilistic}. 

The idea of generalising probability is not new and early works on lower probabilities can be dated to Keynes \cite{keynes1921treatise} and Koopman \cite{koopman1940bases}. In this work, however, we will focus on belief functions, probably the most known generalisation of probability, developed from the works of Dempster \cite{dempster1967upper}  and Shafer \cite{shafer1976mathematical}. 
There is also an extensive work on logic for belief functions: Sossai \cite{sossai1999belief,sossai2001fusion}
Ruspini (\cite{ruspini1986logical} and Smets \cite{smets1988belief,smets1991probability} have developed works in the field. A remarkable point is the similarity between modal epistemic logics and Dempster-Shafer's belief functions, which has been studied since the late 80's \cite{fine1988lower} \cite{boeva1998modelling}\footnote{See \cite{fine1988lower} for a comprehensive description of different works in the area.}. We are mainly interested in the interpretation of belief functions as generalisation of probability, leaving aside the interpretation of Dempster-Shafer theory as a theory of evidence. Some approaches to the topic, such as \cite{flaminio2011reasoning,esteva2000reasoning,hajek2013metamathematics}, \cite{esteva2000reasoning} focus on the uncertainty of fuzzy events. 

\textbf{Goals and structure of this work.}  The purpose of this work is to develop a neighborhood semantics for belief functions that is modally equivalent to Kripke structures with probability, making the latter a subclass of neighbourhood structures. The structure of the paper is as follows: Section \ref{sect-0} offers a summary of the main theoretical concepts in probability that will be used. This work develops a modal logic for probability and belief functions and two semantics for it:  in section \ref{sect1} we introduce the syntax of the language, and then we focus on two possible semantics for it: We first develop a standard Kripke semantics \ref{subsec:KM}, and then a semantics based in neighbourhood semantics \ref{subsec:NM}.  We make a study of several issues related to the structure of the system of neighbourhoods in section \ref{sec:NMC}, which will be used in section \ref{sect3}, where it is studied the correspondence between probabilistic Kripke models and a subclass of belief neighbourhood models, additive ones. It is based on the proofs of the relationship between augmented neighbourhood models and Kripke models. Thus  we will compare them, observing their differences in expressiveness.  Section \ref{secfinal} concludes the paper.

\section{Basic concepts}\label{sect-0}

Subjective probability is an agent's estimate or belief about the likelihood of an event. In this work, we will use propositional formulas to represent events using atomic and compound propositions, while modal operators ranging over propositions will serve to represent the agent's degrees of belief.  Probability is defined over $\sigma$-algebras, but here we will work over Boolean algebras.

\begin{definition}
    \textbf{(Probability measure on a Boolean algebra.)} Given a Boolean algebra $\mathcal{A}$, the  probability measure $pr\colon \mathcal{A}\to [0,1]$  is a function that satisfies the following axioms:
    
\begin{enumerate}
\item $pr(A)\geqslant 0$ for any $A\in\mathcal{A}$.
\item $pr(\top) = 1$.
\item Finite additivity: $pr(\bigcup_{i=1}^n A_i)= \sum_{A_i=1}^n pr (A_i)$, for a family of pairwise disjoint events  $\{A_i\}_{i=1}^n$, $n\in \mathbb{N} $.
\end{enumerate}
\end{definition}

For $\sigma$-algebras in general, the third axiom is stated for countable sets: $pr(\bigcup_{i=1}^\infty A_i)= \sum_{A_i=1}^\infty pr (A_i)$.

Belief functions are a generalisation of subjective probabilities \cite{halpern1992two}. The following definition is taken from \cite{corsi2023logico}, which is an adaptation to Boolean algebras of the  definition in  \cite{shafer1976mathematical}:

\begin{definition}{\textbf{(Belief function on a Boolean algebra.)}} A belief function on a Boolean algebra $A$ is a function $bel\colon A \to [0, 1]$ satisfying: 
\begin{enumerate}
\item $bel(\top)=1$ and $bel(\bot)=0$.
\item $bel(A)\geqslant 0$ for any $A\in \mathcal{A}$.
\item Superadditivity: $bel(\bigcup_{i=1}^n A_n)=\sum_{i=1}^n \sum_{J\subseteq \{1,\ldots,n\}:|J|=i} (-1)^{i+1} bel(\bigcap_{j\in J} A_j)$
\end{enumerate}

\end{definition}

The last condition is also called monotonicity. It can be seen as a weakened form of additivity. Belief functions are a subset of a broader family, lower probabilities. For the purposes of this work, it is sufficient to know that lower probabilities encompass belief functions, that the lower probability of $\top$ is 1, and the lower probability of $\bot$ is 0, whereas monotonicity is not met, but a more general condition instead \cite{halpern1992two}. Lower probabilities and belief functions can be seen as the lower bound of a set of probabilities, that is, stating $bel(A)=\alpha$ entails that all probabilities for $A$ greater than or equal to $\alpha$ are compatible with that belief function. From now on, when the term ``lower probabilities" in mentioned in this work, it actually refers to belief functions as described above.

The following remark, that is a consequence of the previous definitions. will be one of the central ideas of this paper. 

\begin{remark}\label{prandbel}
\textbf{(Probability functions and belief functions.)}
A probability function is an additive belief function (an additive lower probability). \cite{corsi2023logico}
\end{remark}

As we will see, whereas Kripke semantics is able to interpret probabilities directly, in neighbourhood semantics they are a special case of belief functions. Neighbourhood semantics thus offers a broader expressiveness than Kripke semantics.

\section{Syntax of $\mathcal{L}_{PR}$}\label{sect1}

We now develop a modal propositional language to express probabilistic beliefs. It is a variation of works such as \cite{fagin1988reasoning}, \cite{heifetz2001probability}, \cite{aumann1999interactive}, and other similar works, although with a  different notation. 
 \texttt{At} will denote a non-empty finite set of atomic propositions.

\begin{definition} \textbf{(Syntax.)} We define recursively the formulas of the language $\mathcal{L}_{PR}$ as follows, where $p\in \texttt{At}$:

\[  \phi ::= p\ |\ \neg\phi\ |\ (\phi\lor\phi )\ |\ B_{\geqslant \alpha} \phi\ |\  \underline{B}_{\geqslant\alpha}\phi \]
\end{definition}

The usual equivalences hold: $\phi\land\psi$ stands for $\neg (\neg\phi\lor\neg\psi)$, and $\top$ and $\bot$ are for $\phi\lor\neg\phi$ and $\phi\land\neg\phi$. 
 $\alpha\in [0,1]$ is a rational value. $B_{\geqslant \alpha}\phi $ is read in the same way that $L_\alpha$ in \cite{heifetz2001probability}: ``\emph{ the agent assigns to $\phi$ a probability equal to or greater than $\alpha$ (at least $\alpha$)"}. Similarly $B_{\leqslant\alpha}\phi$ is shorthand for $B_{\geqslant 1-\alpha}\neg\phi$, which is read as ``\emph{the agent assigns to $\phi$ a probability at most $\alpha$}". $B_{=\alpha}\phi$, ``\emph{the agent assigns to $\phi$ a probability of (exactly) $\alpha$}", is just $B_{\geqslant\alpha}\phi\land B_{\leqslant\alpha}\phi$. $B_{<\alpha}\phi$ is $\neg B_{\geqslant \alpha}\phi$  and $B_{>\alpha}\phi$ is $\neg B_{\geqslant 1-\alpha}\neg\phi$.

Analogously, $\underline{B}_{= \alpha}\phi$ refers to a belief function, which can be understood as a lower probability, and could be read as  ``\emph{the agent assigns a lower probability   of $\alpha$ to $\phi$}", or more informally as ``\emph{the agent considers that the probability of $\phi$ is at least $\alpha$}" (that is, $\alpha$ is the lower probability of $\phi$). $\underline{B}_{\geqslant \alpha}\phi$ has a rather complicated translation, as ``\emph{the agent assigns a lower probability to $\phi$, which is at least $\alpha$}".  $\underline{B}_{\leqslant \alpha}\phi$, $\underline{B}_{> \alpha}\phi$ and  $\underline{B}_{< \alpha}\phi$ are defined as their probability counterparts.

\section{Semantics of $\mathcal{L}_{PR}$}\label{sec:sem}

\subsection{Kripke semantics}\label{subsec:KM}

Works on Kripkean semantics for probabilistic modal logic such as \cite{shirazi2007probabilistic} define a probability distribution over the worlds that are accessible from a certain world $w$, and it is the type of semantics that we will develop here. Fagin and Halpern \cite{fagin1990logic} provide another approach, a probability distribution defined for measurable \emph{sets} of accessible worlds in a Kripke model.

\begin{definition} 
\textbf{(Probabilistic Kripke model.)}\label{prm-k} A probabilistic Kripke model  $\mathcal{M}^P_K$ is  a tuple $\langle W, \mu,R,v \rangle$, where $W$ is a non-empty and finite set of worlds or states, $v\colon \texttt{At}\to \wp(W)$ a valuation function that assigns to each letter the set of worlds in $W$ that satisfy it, $\mu:W\times W\to [0,1]$  a probability function following these rules for a given $w$:

\begin{enumerate}

    \item $\mu (w,w')\geqslant 0$ 

    \item $\sum_{w_i\in W} \mu(w,w_i) = 1$
\end{enumerate}

$R\colon W\times W$ is defined in this way: $(w,w')\in R$ (usually stated as $wRw'$) iff $\mu(w,w')>0$, and $(w,w')\notin R$ iff $\mu(w,w')=0$. 
\end{definition}

The definition of $R$ and the requirement that $\sum_{w_i\in W}\mu(w,w')=1$  entail that for every $w$ there is at least one $w'$ such that $wRw'$. Neither $R$ nor $\mu$ are symmetric: In general, $\mu(w,w')\neq \mu(w',w)$ and $wRw'$ does not equate to $w'Rw$.

\begin{remark}\label{accessK} 
\textbf{($R$ and $\mu$:)}
The accessibility relationship $R$ is not strictly necessary, since it is implicit in the probability measure $\mu$. We have included it for simplifying some explanations.
\end{remark}

\begin{definition} \textbf{(Probability and belief measures in Kripke frames.)}
Given a $X\subseteq W$, we define its probability and belief measures from world $w\in W$ as follows:

\[  pr(w,X)=bel(w,X)=\sum_{w_i\in X} \mu(w,w_i) \]
\end{definition}

A pointed model is a pair $\langle M,w\rangle$ with $M$ a model and $w$ an world of its domain $W$. We will write $M,w\models_K \phi$ if the formula $\phi$ is true at the world $w\in W$ in Kripke model $M$. We assign truth values to formulas at a certain world as follows:

\begin{definition} \textbf{(Interpretation of $\mathcal{L}_{PR}$ in probabilistic Kripke models.)}
Given a probability Kripke model $\mathcal{M}^P_K$, the satisfaction relation $\models_K$ between pointed models and formulas is defined as follows:

\begin{enumerate}
\item $\mathcal{M}^P_K,w\models_K p$ iff $w\in v(p)$.
\item $\mathcal{M}^P_K,w\models_K \neg\phi$ iff $\mathcal{M}^P_K,w\not\models\phi$.
\item $\mathcal{M}^P_K,w\models_K (\phi\lor\psi)$ iff $\mathcal{M}^P_K,w\models\phi$ or $\mathcal{M}^P_K,w\models\psi$.
\item $\mathcal{M}^P_K,w\models_K B_{\geqslant\alpha}\phi$ \ iff $pr(w,\llbracket \phi\rrbracket_M)\geqslant \alpha$
\item  $\mathcal{M}^P_K,w\models_K \underline{B}_{\geqslant\alpha}\phi$ iff $bel(w,\llbracket \phi\rrbracket_M)\geqslant \alpha$
\end{enumerate}
\end{definition}

 $\llbracket\phi\rrbracket_{\mathcal{M}^P_K}$ denotes the set $\{ w\in W\ |\ M,w\models_K \phi\}$. 
 We simply use  $\llbracket\phi\rrbracket_M$ or $\llbracket\phi\rrbracket$ if there is there no possibility of confusion.
 The interpretation of the probabilistic modal operator we have just defined is close to previous works in subjective probabilities (\cite{shirazi2007probabilistic} for example), and formally it is close to probabilistic labelled transition systems \cite{LARSEN19911}.

\begin{definition}\label{def:leqKM}
\textbf{(Logical consequence and logical equivalence in probabilistic Kripke models.)}
A formula $\psi$ is a logical consequence of $\phi$  in probabilistic Kripke models if for all pointed models such that $\mathcal{M}^P_K,w\models\phi$, it is also that  $\mathcal{M}^P_K,w\models\psi$. $\phi$ and $\psi$ are logically equivalent in probabilistic Kripke models provided that,  for all pointed models, $\mathcal{M}^P_K,w\models\phi$ if and only if $\mathcal{M}^P_K,w\models\psi$.
\end{definition}

\begin{remark} \textbf{(Events and propositions.) }
We can identify each event with the set of all logically equivalent propositions, instead of with propositions. Thus, if $\phi$ and $\psi$ are logically equivalent, for any Kripke model and any world in it, $\mathcal{M}^P_K,w\models B_{\geqslant \alpha}\phi$ if and only if $\mathcal{M}^P_K,w\models B_{\geqslant \alpha}\psi$. 
\end{remark}

\begin{remark}\label{pr-bel-KM} \textbf{(Equivalence between probabilities and belief functions in $\mathcal{L}_{PR}$.) } Given that the conditions for the satisfaction of $B_{\geqslant\alpha}\phi$ and $\underline{B}_{\geqslant\alpha}\phi$ are the same, they are logically equivalent. This is not surprising, since probabilities are a kind of belief functions \ref{prandbel}. However, we cannot adequately interpret a formula whose interpretation is a belief function which is not a probability. Since we are going to study the relationship between Kripke and neighbourhood semantics for $\mathcal{L}_{PR}$, we have included the operator $\underline{B}_{\geqslant\alpha}$ in order to ensure that both semantics interpret all the formulas of the same language. 
\end{remark}

\begin{example}{(Example 1: A simple probabilistic Kripke model.)}\label{example1}

Let us define $W=\{w_1,w_2,w_3,w_4\}$. $\mu\colon W\times W$ is defined in the following matrix, where $M(\mu)_{i,j}=\mu (w_i,w_j)$ (for example, $\mu (w_1,w_2)=M(\mu)_{1,2}=0.4$):

\setlength{\arraycolsep}{8pt}

\[
M(\mu)_{ij}= \footnotesize{\begin{pmatrix}
0 & 0.4  & 0.6 & 0 \\
0.4 & 0  & 0.6 & 0\\
0 & 0  & 0.1 & 0.9 \\
0 & 0 & 0 & 1 \\
\end{pmatrix}
}
\]

It can be verified that the sum of each row is 1.  We add a valuation function such that $v(p)=\{w_1,w_3\}$ and $v(q)=\{w_1,w_2\}$. In the picture belowe, we have omitted the relationships when $\mu$ is 0.

\begin{center}
\begin{tikzpicture}

\node  [draw,circle] (a) at (0,0)   {\small $w_1$};
\node  [draw,circle] (b) at (2,0) {\small $w_2$};
\node  [draw,circle] (c) at (0,-2) {\small $w_3$};
\node [draw,circle] (d) at (2,-2) {\small $w_4$};

\node [anchor=west, align=left] at (-1,0) {\small $p,q$};
\node [anchor=west, align=left]  at (-1,-2) {\small $p$};
\node [anchor=west, align=left] at (2.5,0) {\small $q$};

\draw[->] (a) to[out=30, in=150, looseness=1] node[above] {\small $0.4$} (b);
\draw [->] (a) -- node[left] {\small $0.6$}  ++(c);

\draw[->] (b) to[in=330, out=210, looseness=1] node[below] {\small $0.4$} (a);
\draw [->] (b) -- node[right] {\small $0.6$}  ++(c);

\draw [->] (c) -- node[below] {\small $0.9$}  ++(d);

\draw[->] (c) to[out=240, in=210, looseness=8] node[left] {\small $0.1$} (c);
\draw[->] (d) to[out=330, in=300, looseness=8] node[right] {\small $1$} (d);

\end{tikzpicture}

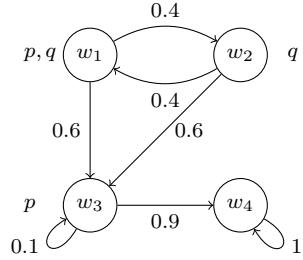
\textit{Picture 1. Example of probabilistic Kripke model.}
\end{center}

We have that $R(w_1)=\{w_2,w_3\}$. In this model, $M,w_1\models_K B_{\geqslant 0.6} p$ and  $M,w_1\models_K B_1 B_{\geqslant 0.1} p$. 

\end{example}

\subsection{Neighbourhood semantics}\label{subsec:NM}

\begin{example}{\textbf{(Example 2: Expressing bounded probabilistic beliefs.)} }

Probability modal logic is commonly used for expressing beliefs with a certain degree, but its formulation in natural language may be misleading. For example, if $r$ means that it will rain tomorrow, $B_{\geqslant 0.5} r$ is understood as  ``Ann believes that the probability of raining tomorrow is at least 50\%  (i.e. 0.5)", where the modal operator is used to express the graded beliefs of the agent we are considering, namely Ann. However,  the interpretation of the modal formula according to the previous definition is not exactly what has been expressed in natural language.  $B_{\geqslant 0.5} r$ actually means that Ann believes with a certain degree of confidence that it will rain tomorrow, and that degree is at least 0.6: There is a definite probability value, which is at least 0.5.

``Ann believes that the probability of raining tomorrow is at least 50\%" is usually interpreted in natural language as a lower probability: Ann is considering that the probability of raining cannot be lower than 50\%, but she does not have a definite value, and therefore all values above 50\% are compatible with her belief. If we see the model in Example 1, $w_1$ satisfies $B_{\geqslant 0.5} r$. If we try to express a lower probability by stating $B_{\geqslant 0.5} r\land \neg B_{> 0.5} r$, $w_1$ does not satisfy that formula, which actually expresses that the agent believes that the probability of raining is exactly 50\%. Thus, we need another semantics for interpreting lower probabilities.
\end{example}

A possible alternative to Kripke semantics is neighbourhood semantics. It is better suited for non-normal logic, and probabilistic modal logic is a monotonic, non-normal modal logic. Neighbourhood semantics for probability modal logic are much rarer than Kripke semantics: Arló-Costa \cite{arlo2005non} proposed a model for a non-normal high probability operator,  and Herzig \cite{herzig2003modal}  offered a qualitative interpretation of the modal operator.

We briefly recall a few preliminary definitions of neighbourhood semantics:

\begin{definition}
{\textbf {(Neighbourhood frame and neighbourhood model.)
}}\label{def:-Neigh-frame-and-model}\index{Neighbourhood frame and model}
A \emph{neighbourhood frame} is the ordered pair $\langle W,N\rangle$
formed by a non-empty and finite set $W$ and a neighbourhood function over it
$N:\ W\to\wp(\wp(W))$ that assigns a set of subsets of $W$ to each $w\in W$. Given a neighbourhood frame $\langle W,N\rangle$
and a set of atomic propositional formulas \texttt{At} with a valuation
function $v:\ \texttt{At}\to\wp(W)$, a \emph{neighbourhood model}
is a triple $M=\langle W,N,v\rangle$.
\end{definition}

In neighbourhood semantics, $M,w$ satisfies $\square\phi$ if $\llbracket\phi\rrbracket_M\in N(w)$, that is, if the set formed by all worlds that satisfy $\phi$ is one of the sets in the neighbourhood of $w$. 

\begin{definition}
\textbf{(Some properties of neighbourhoods.)}
These are some properties that neighbourhoods may have:
\begin{enumerate}
\item A neighbourhood is \emph{monotonic} if $X\in N(w)$ and $X\subseteq Y$ entails $Y\in N(w)$.
\item A neighbourhood is closed under finite intersections provided that for a family of sets $\{X_i\}_{i\in J}$ such that for each $i\in J$ ($J$ finite), $X_i\in N(w)$, then $\bigcap \{X_i\}_{i\in J}\in N(w)$.
\item A neighbourhood is closed under finite unions provided that for a family of sets $\{X_i\}_{i\in J}$ such that for each $i\in J$ ($J$ finite), $X_i\in N(w)$, then $\bigcup \{X_i\}_{i\in J}\in N(w)$.
\item A neighbourhood is closed under complement provided that for each $X\in N(w)$, $X^C\in N(w)$.
\item A neighbourhood $N(w)$ \emph{contains its core} $\cap N(w)$ if the core of the neighborhood, the intersection of all sets in $N(w)$, is a set in $N(w)$.
\item A monotonic neighbourhood that contains its core is said to be \emph{augmented}.
\end{enumerate}
\end{definition}

A high-probability modal operator is monotonic ($P(\phi\land\psi)\to P\phi\land P\psi$ / $P\phi\lor P\psi\to P(\phi\lor\psi)$  is a valid axiom scheme), where ``high probability" refers to a probability higher than a given threshold. However, this axiom can be generalised for any probability,
We will consider a belief function operator, intended to be a lower probability operator. To define the neighbourhood semantics of our modal probabilistic operator, we will add a probabilistic metrics. We will combine the previous idea of the high probability operator with a set of nested neighbourhoods, akin to Lewis' sphere model in \cite{lewis1973}, one for each probability. 
We will also add a belief function.

\begin{definition} {\bfseries (Belief neighbourhood   model.)} A tuple $\mathcal{M}^B_N=\langle W,b,$ $\{N_{\geqslant \alpha}\}_{\alpha\in [0,1]},$ $\{N_{> \alpha}\}_{\alpha\in [0,1)},$ $v\rangle$ is  a belief neighbourhood   model if $W$ is a non-empty, finite set of worlds, $N$ a neighbourhood function $W\to \wp(\wp(W))$, $b$ a belief function $W\times \wp(W)\to [0,1]$, and $v\colon \texttt{At}\to W$ a valuation function.
\end{definition}

\begin{definition} \textbf{(Belief function in $N(w)$.)}
The function $b:W\times\wp(W)\to [0,1]$ assigns a value to each $X\in N_{\geqslant 0}(w)$,  satisfying:
\begin{enumerate}
\item $b(w,W)=1$,\ $b(w,\varnothing)=0$.
\item $bel(w,\bigvee_{i=1}^n \phi_n)\geqslant \sum_{i=1}^n \sum_{J\subseteq \{1,\ldots,n\}:|J|=i} (-1)^{i+1} bel(w,\bigwedge_{j\in J} \phi_j)$
\end{enumerate}
\end{definition}

\begin{definition} \textbf{(Monotonicity of in $N_{\geqslant \alpha}(w)$.)} \label{monot-n}
All $N_{\geqslant \alpha}(w)$ are monotonic: If $X\in N_{\geqslant \alpha}(w)$ and $X\subseteq Y$, then $Y\subseteq N_{\geqslant \alpha}(w)$.
\end{definition}

\begin{definition} \textbf{($N_{\alpha}(w)$ and $b(w,X).)$}
For all $w\in W$, if $b(w,X)\geqslant \alpha$, then $X\in N_{\geqslant \alpha}(w)$, and if $b(w,X)> \alpha$, then $X\in N_{> \alpha}(w)$.
\end{definition}

\begin{proposition}
\textbf{(Nesting of neighbourhoods.)}  From the previous definitions, the system of nested neighbourhoods is established in this way: 

\begin{enumerate}
    \item If $\alpha\geqslant \beta$, then $N_{\geqslant \alpha}(w)\subseteq N_{\geqslant\beta}(w)$.
    \item $N_{> \alpha}(w)\subseteq N_{\geqslant\alpha}(w)$.
    \item If $\alpha>\beta$, then $N_{\geqslant \alpha}(w)\subseteq N_{>\beta}(w)$.
\end{enumerate}
\end{proposition}

\begin{remark}\label{accessN}
\textbf{(A remark about the relationship between $N_{\geqslant\alpha}(w)$ and $b(w,X)$.)}
Analogously to \ref{accessK}, the system of nested neighbourhoods $N_{\geqslant\alpha}$ is not strictly necessary. We could develop it just by using the belief measure $b$. If we did so, however,  we should redefine   neighbourhood semantics' concepts in terms of a \emph{neighbourhood measure}, and all the concepts we are going to introduce here should be formulated in those terms.
\end{remark}

Since $b(w,\varnothing)=0$ for all $w\in W$, then $\varnothing\in N_{\geqslant 0}(w)$. Therefore, all subsets of $W\in N_{\geqslant \alpha}(w)$.  In practice, we are only interested in the sets that correspond to the truth set of some $\phi$. Thus, for any $w\in$ and any formula $\phi$, $\llbracket\phi\rrbracket_M\in N_{\geqslant 0}(w)$. If $\beta =b(w,\llbracket\phi\rrbracket_M)$, then $\llbracket\phi\rrbracket_M\in N_{\geqslant \alpha}(w)$ for all $\alpha\leqslant \beta$. In order to do that, we must define the semantics of the model.

\begin{definition} \textbf{(Semantics in belief neighbourhood models.)}

Given a belief neighbourhood model $\mathcal{M}^B_N$, the satisfaction relation $\models_N$ between pointed models and formulas is defined as follows:

\end{definition}
\begin{enumerate}
\item $\mathcal{M}^{ B}_{N},w\models p$ iff $w\in v(p)$.
\item $\mathcal{M}^{ B}_{N},w\models \neg\phi$ iff $\mathcal{M}^{ B}_{N},w\not\models\phi$.
\item $\mathcal{M}^{ B}_{N},w\models (\phi\lor\psi)$ iff $\mathcal{M}^{ B}_{N},w\models\phi$ or $\mathcal{M}^{ B}_{N},w\models\psi$.

\item $\mathcal{M}^{ B}_{N},w\models \underline{B}_{\geqslant \alpha}\phi$  iff $\llbracket\phi\rrbracket \in N(w)$ and $b(w,\llbracket\phi\rrbracket)\geqslant \alpha$.
\item $\mathcal{M}^{ B}_{N},w\models \underline{B}_{> \alpha}\phi$  iff $\llbracket\phi\rrbracket \in N(w)$ and $b(w,\llbracket\phi\rrbracket)> \alpha$.

\vspace{2mm}

The following definition derives from the previous one:

\vspace{2mm}

\item $\mathcal{M}^{ B}_{N},w\models B_{\geqslant \alpha}\phi$ iff $\llbracket\phi\rrbracket \in N(w)$ and $b(w,\llbracket\phi\rrbracket)\geqslant \alpha$ and $b(w,\llbracket\neg\phi\rrbracket)_M+b(w,\llbracket\phi\rrbracket)_M=1$.
\end{enumerate}

A special case of the latter is $\mathcal{M}^{ B}_{N},w\models B_{= \alpha}\phi$, which is $b(w,\llbracket\phi\rrbracket)_M=\alpha$.  $b(w,\llbracket\neg\phi\rrbracket)_M=1-\alpha$ according to the last definition.

 $\llbracket\phi\rrbracket_{\mathcal{M}^B_N}$ denotes the set $\{ w\in W\ |\ M,w\models_N \phi\}$. Again, we will denote it by  $\llbracket\phi\rrbracket_M$ or $\llbracket\phi\rrbracket$ is there no possibility of confusion about the semantics we are using.

\begin{definition}\label{def:leqNM}
\textbf{(Logical consequence and logical equivalence in belief neighbourhood models.)}
A formula $\psi$ is a logical consequence of $\phi$ if for all belief neighbourhood pointed models $\mathcal{M}^B_N,w$ that satisfy $\phi$, it is also that  $\mathcal{M}^B_N,w\models_K\psi$.
$\phi$ and $\psi$ are logically equivalent provided that, for all pointed belief neighbourhood models, $\mathcal{M}^B_N,w\models_K\phi$ if and only if $\mathcal{M}^B_N,w\models_K\psi$.

\end{definition}

\begin{remark}\label{byp-N} \textbf{(Probabilities and belief functions in neighbourhood models.) } In neighbourhood semantics, $\underline{B}_{\geqslant\alpha}\phi$ is a logical consequence of $B_{\geqslant\alpha}\phi$. The converse is not true, so therefore they are not logically equivalent.

\end{remark}

\begin{proposition}\label{propertiesn}
Some properties of the set of nested neighbourhoods in the belief neighbourhood model are the following:

\begin{enumerate}
    \item $N_{\geqslant 0}(w)$ and $N_{1}(w)$ are the only neighbourhoods that contain their respective cores ($\varnothing$ for $N_{\geqslant 0}(w)$). Since all neighbourhoods are monotonic, $N_{\geqslant 0}(w)$ and $N_{1}(w)$ are augmented.\label{prop:N0-1aug}
    \item  $N_{\geqslant 0}(w)$ and $N_{1}(w)$ are the only neighbourhoods that are closed under intersection, complement and union.
    \item $N_{\geqslant 0}(w)$ is also the only one in which  $W$ and $\varnothing$ belong to the neighbourhood.
\end{enumerate}
\end{proposition}

\begin{example}{(Example 2: A belief neighbourhood model.)}\label{example2}

$W=\{w_1,w_2,w_3,w_4\}$. In the following table, we show the truth sets in $N_{\geqslant 0}(w_1)$ and their belief value. For readability, we write $b(X)$ instead of $b(w_1,X)$. We define just the neighbourhood of $w_1$:

\begin{center}
{\small
\begin{tabular}{c|c|c|c|c}

$b(w_1,W)$ = 1  &   & &  &  \\
$b(W\backslash \{w_1\})$ = 0.9  & $b(w_1,W\backslash \{w_2\})$ = 0.7   & $b(W\backslash \{w_3\})$ = 0.6  & $b(W\backslash \{w_4\})$ = 0.6  \\
$b(\{w_1,w_2\})$ = 0.2  & $b(\{w_1,w_3\})$ = 0.3  & $b(\{w_4\})$ = 0.4 &  &  \\
$b(\{w_2,w_3\})$ = 0.4  &  $b(,\{w_2,w_4\})$ = 0.5 & $b(w_3,\{w_4\})$ = 0.6 &  &  \\
$b(\{w_1\})$ = 0.1  & $b(\{w_2\})$ = 0.1  & $b(\{w_3\})$ = 0.2& $b(\{w_4\})$ = 0.3 &  \\
$b(\{ \})$ = 0  &   & &  &   \\
\end{tabular}
}

\vspace{2mm}

\textit{Table 1.} Example of neighbourhood frame (only $b(w_1,X)$).
\end{center}

We can verify that $b$ is superadditive: For example, $b(w_1,\{w_1,w_2,w_3,w_4\})=1\geqslant b(w_1,\{w_1,w_2\})+b(w_1,\{w_3,w_4\})=0.2+0.6$. 

If we add a valuation, we obtain a model. Let us make $\{w_3,w_4\}=\llbracket p\rrbracket_M$. $b(w_1,\llbracket p\rrbracket_M)=0.7$, and we have that $M,w_1\models \underline{B}_{0.7}p$.

\end{example}

\section{The non-monotonic core of $N_{> 0}(w)$}\label{sec:NMC}
This section provides a characterisation of the sets in the nested neighbourhoods in terms of disjoint unions of certain sets that we will call elementary sets. In particular, we will see the relationship between the core of $N_1(w)$ and these sets, which are the sets in $N_{\geqslant 0}(w)$ that are closed under the inclusion relation in that neighbourhood. This will be especially useful in the next section for the proofs of the equivalence between a certain subclass of additive belief neighbourhood models and probabilistic Kripke frames. The following two definitions are taken from \cite{pacuit2017neighborhood}:

\begin{definition}
\textbf{(Non-monotonic core.)}
The non-monotonic core of a neighbourhood $N(w)$, denoted $N(w)^{\mathcal{NC}}$, is a subset of $N(w)$ defined as follows:
 \[ N(w)^{\mathcal{NC}} =\{ X\in N(X)\ |\ \text{For all } X'\subseteq W \text{, if }X'\subset X, X'\notin N(w)\} \]
\end{definition}

$N(w)^{\mathcal{NC}}$ contains the subset of minimal elements in $N(w)$ under the subset relationship.  Regarding our system of nested neighbourhoods, we have the following: From the definition of non-monotonic core, it is immediate that  $N_{\geqslant 0}(w)^{\mathcal{NC}}$ contains just the empty set, whose belief measure is 0, and  $N_1(w)^{\mathcal{NC}}$ contains just $\cap N_1 (w)$.

\begin{definition}
\textbf{(Core complete.)}
A monotonic neighbourhood $N(w)$ is core-complete if for all $X\in N(w)$, there is some $X'\in N(w)^{\mathcal{NC}}$ such that $X'\subseteq X$.
\end{definition}

A neighbourhood is core-complete if every set in it has a subset in the non-monotonic core of the neighbourhood. If a certain $N(w)$ is monotonic and contains a finite number of sets, it is core-complete \cite{pacuit2017neighborhood}. 

 In this work, we are dealing with a finite $W$, so the previous comment applies in particular to all nested neighbourhoods:

 \begin{proposition}
All nested neighbourhoods $N_{\geqslant \alpha)}(w)$ ($\alpha\in [0,1]$) and $N_{> \alpha)}(w)$ ($\alpha\in [0,1)$) are core-complete.
\end{proposition}
 
 We are going to use the previous concepts to study the core of $N_1(w)$ in terms of the sets in the non-monotonic core of $N_{> 0}(w)$.

$N_{\geqslant 0}(w)^{\mathcal{NC}}$ contains just one set, the empty set. $N_{> 0}(w)^{\mathcal{NC}}$ is the subset of $N_{> 0}(w)$ formed by the sets in $N_{\geqslant 0}(w)$ whose belief function is greater than 0 and do not have subsets. We will call the sets in $N_{> 0}(w)^{\mathcal{NC}}$ \emph{elementary sets}.

\begin{lemma}
The elementary sets in $N_{> 0}(w)^{\mathcal{NC}}$ are pairwise disjoint.
\end{lemma}
\begin{proof}
The sets in a non-monotonic core are closed under the subset relation.
\end{proof}

The following lemma is immediate, given that $N_{>0}(w)$ is core-complete: 

\begin{lemma}\label{core-element-1}
Any $X\in N_{> 0}(w)$, have as a subset  the disjoint union of some elementary sets in $N_{> 0}^{\mathcal{NC}}(w)$.
\end{lemma}

We introduce the following definition:
 
\begin{definition}
\textbf{(Interior of a set in $N_{\geqslant 0}(w)$.)}\label{def:interior}
We define the interior of a set $X\in N_{\geqslant 0}(w)$ as the union of all the sets in the maximal set of elementary sets of $X$, that is, the set that contains all elementary sets that are subsets of $X$:

    \[\mathtt{int}_N (X) = \bigcup_{E_i\subseteq X, E_i \in N_{>0}(w)^{\mathcal{NC}}} E_i\]
\end{definition}

We have chosen the name by analogy with the concept of interior in topology, although the concept is not the same, since the elementary sets mentioned here are not open spaces.
This maximal subset will allow us to define $b(w,X)$ in terms of the belief functions of the elementary sets in the maximal subset, given that the belief function is superadditive, and the elementary sets
are pairwise disjoint. 

\begin{proposition}\label{bel-X}
\textbf{(Belief function of a set $X\in N_{>0}(w)$.)}
Given some $X\in N_{>0}(w)$, $b(w,X)\geqslant \mathtt{int}_N (X) $
\end{proposition}

If the maximal subset is empty, $b(w,X)=0$. 

\begin{proposition}\label{decomp-core}
$\cap N_1 (w)$ is equal to the disjoint union of all elementary sets in $N_{> 0}(w)^{\mathcal{NC}}$.
\end{proposition}
\begin{proof}
By lemma \ref{core-element-1}, $\cap N_1(w)$ has a  a maximal subset of elementary sets, whose union is $\texttt{int}_N(\cap N_1(w))$. 
We must prove that all $w_i\in \cap N_1(w)$ belongs to some elementary set. If this were not the case, then $x_i\in A$ for some set $A$ such that $b(w,A)=0$. In this case $b(w,\cap N_1(w)\backslash A)=1$, and being disjoint $A$ and $\cap N_1(w)\backslash A$, then $b(w,A)+b(w,\cap N_1(w)\backslash A)=1$. 
but this would mean that $\cap N_1(w)\subseteq (\cap N_1(w)\backslash A)$  and $(\cap N_1(w)\backslash A)\subseteq \cap N_1(w)$, that is, $\cap N_1(w)=\cap N_1(w)\backslash A$. Therefore, $A$ is empty and all worlds in $\cap N_1(w)$ belong to some elementary set.  Now we must prove that all elementary sets are subsets of $\cap N_1(w)$. Let us suppose that there is an elementary set $E_i\not\subseteq \cap N_1(w)$. We have that  $b(w,\cap N_1(w)\cup E_i)\geqslant b(w,\cap N_1(w))+b(w,E_i)>1$, because the sets are disjoint. Therefore, all elementary sets are subsets of $\cap N_1(w)$.

All elementary sets are pairwise disjoint, so $\cap N_1(w)$ is the pairwise union of all elementary sets:  $\cap N_1(w)=\texttt{int}_N(\cap N_1(w))=\texttt{int}_N(W)$.
\end{proof}

\begin{proposition}\label{core-element-2}
\textbf{(Belief function of the maximal subset of elementary sets of   $N_1(w)$.)}
The sum of the belief functions of all elementary sets is  equal to or less than 1.
\end{proposition}
\begin{proof}
From the previous propositions and lemmas, belief function being superadditive and $b(w,\cap N_1(w))=1$.
\end{proof}

Lastly, we are going to characterise the class of belief neighbourhood frames in which beliefs are additive.

\begin{definition}\label{additN}
\textbf{(Additive belief neighbourhood frame.)}
A belief neighbourhood frame  is additive if the following holds:
\[  \text{For all } w\in W, b(w,\bigcup \{A_i\}_{i\in J})= \sum_{i\in J} b(w,A_i) \text{ (All }A_i \text{pairwise disjoint)} \]
\end{definition}

In this case, the belief function reduces to a probability function. To prove that a neighbourhood is additive, we do not need to verify it for all possible sets in  $N_{\geqslant 0}(w)$. We just need to verify the following proposition.

\begin{proposition}\label{moon-zappa}
\textbf{(Core complete neigbhourhood of $N_{>0}(w)$ and additive neighbourhoods.)}
If and only if the frame is additive, then for all $w\in W$ and, for all $X\in N_{> 0}(w)$, $ \sum_{E_i\in N_{> 0}(w)^{\mathcal{NC}}}$ $b(w,E_i)=1$. 
\end{proposition}
\begin{proof}
If the frame is additive, by definition of additivity,  proposition \ref{core-element-2} and the elementary sets being pairwise disjoint, the sum of all the elementary sets in the neighbourhood of all $w\in W$ is 1. If the frame is not additive, for some $w\in W$ the sum of the belief functions of the elementary sets in its neighbourhood is less than 1.
\end{proof}

We need to provide a condition for a set to be measurable. 

\begin{definition}
\textbf{(Well-defined set.)}
A set $X$ in $N_{\geqslant\alpha}(w)$ is well-defined if $X\cap \cap N_1(w)=\texttt{int}_N(X)$.
\end{definition}

In other words, a set is well-defined if it does not insersect any elementary set which is not its subset.  A consequence of this definition, the last proposition and proposition \ref{bel-X} is the following:

\begin{lemma}\label{b-X-aa}
\label{b-X-1}If the frame is additive, if a set $X$ in $N_{\geqslant\alpha}(w)$ is well-defined, then.  $b(w,X)=b(w,\texttt{int}_N(X))$.
\end{lemma}
 
If the frame is not additive or the set is not well-defined, we only can state $b(w,X)\geqslant b(w,\texttt{int}_N(X))$. When we define a model, sets corresponding to the truth sets of formulas should be well-defined. Otherwise, they may not be measurable.

\section{Relationship between probability and belief models}\label{sect3}

The relationship between probabilistic Kripke semantics and belief neighbourhood semantics is akin to the relationship between Kripke semantics and neighbourhood semantics: a Kripke model has a kind of equivalence called \emph{modal equivalence} with an augmented neighbourhood model. 

We will ground our proof for probabilistic models in the aforementioned  equivalence between Kripke semantics and a subclass of the neighbourhood semantics. First, we recall the latter proof, and then we will develop our own proof for probability and belief. The original proof for the relationship between Kripke and augmented neighbourhood models can be found in \cite{chellas1980modal}. Here we will follow Pacuit \cite{pacuit2017neighborhood}, which is also the source of these definitions:

\begin{definition}
\textbf{(R-necessity.)}
Given a Kripke frame $\langle W,R\rangle$, $X\subseteq W$ is  R-necessary at $w$ if $R(w)\subseteq W$. The set $\{ X\subseteq W\ |\ R(w)\subseteq X  \}$ will be denoted by $\mathcal{N}^R_w$. 
\end{definition}

\begin{definition}
\textbf{(Pointwise equivalence.)}
Given a non-empty set $W$, a Kripke frame $\langle W,R\rangle$ and a neighbourhood frame $\langle W,N\rangle$ are pointwise equivalent if, for all $w\in W$ and $X\subseteq W$, $X\in N(w)$ iff $X\in \mathcal{N}^R_w$.
\end{definition}

\begin{definition}
\textbf{(Modal equivalence.)} Given a modal language $\mathcal{L}$ and two classes of models for it $M$ and $M'$, we say that $M,w$ is $\mathcal{L}$-modally equivalent to $M,w'$ iff the set $\{\phi\in\mathcal{L}\ |\ M,w\models \phi\}$ is the same than $\{\phi\in\mathcal{L}\ |\ M',w'\models \phi\}$.

Given models $M$ and $M'$, if for every pointed model $M,w$, there is some $w'$ such that $M,w$ and $M,w'$ are pointwise equivalent, then  $M$ and $M'$ are modally equivalent. More generally, a class of models $M$ is modally equivalent to a class of models $M'$ iff for each pointed model $M,w$ there is another pointed model $M',w'$ such that they are modally equivalent.
\end{definition}

Modal equivalence is a concept which encloses other well-known ones, such as bisimulation. However, bisimulation is defined within structures of the same kind (for example, bisimulation in Kripke or neighbourhood models). A usual way of proving that $M$ and $M'$ are modally equivalent  is showing a way of transforming  $M$ into $M'$ and vice versa. The following propositions are also taken from \cite{pacuit2017neighborhood}:

\begin{proposition}\label{kmandanm}
\textbf{(Kripke models and augmented neighbourhood models.)} Let $\langle W,R\rangle $ be a Kripke frame. Then, there is a modally equivalent augmented neighborhood frame. 
\end{proposition}
\begin{proof}
For each $w\in W$, let $N(w)=\mathcal{N}^R_w$ as described above, and $\langle W,R\rangle$ is the desired neighbourhood frame, which is augmented. 
\end{proof}

\begin{proposition}
\textbf{(Kripke models and augmented neighbourhood models.)} 
Let $\langle W,N\rangle$ be an augmented neighborhood frame. Then, there is a modally equivalent relational frame. 
\end{proposition}
\begin{proof}
From an augmented $\langle W,N\rangle$, we define a binary relation $R_N\colon W\times W$ in the following way: For every $w,w'\in W$, let $wR_Nv$ iff $v\in \cap N(w)$. The frame $\langle W,R_N\rangle$ is the desired Kripke frame.
\end{proof}

As a result, we have the following:

\begin{proposition}\label{class-aug-k}
\textbf{(Kripke models and augmented neighbourhood models.)} 
The class of Kripke models is modally equivalent to the class of augmented neighbourhood models.
\end{proposition}
\begin{proof}
If we add a valuation function $v:\texttt{At}\to W$ to the previous structures, we obtain the corresponding Kripke and neighbourhood models. We observe that the satisfaction of formulas is maintained in all $w\in W$, by induction:
\begin{itemize}
    \item The satisfaction of atomic formulas is the same, since $v$ is the same.
    \item The satisfaction of the conjunction and negation of formulas is preserved, since in both semantics their interpretation is the same.
    \item Satisfaction of the modal formulas: 
    \begin{itemize}
        \item If $M,w\models_K \square\phi$, all worlds $w_i$ such that $wRw_i$ satisfy $\phi$. In the augmented neighbourhood model, all $w_i$ are in $\cap N(w)\in N(w)$ and $\cap N(w)\subseteq \llbracket\phi\rrbracket_{M'}\in N(w)$, and thus $M',w\models_N \phi$ by the definition of satifaction of $\square\phi$ in neighbourhood semantics.
        \item If $M',w\models_N \square\psi$, by defining a relation $R_N\colon W\times W$ as stated before, we obtain a set of worlds $w_i$ such that $wR_nw_i$, and all of them satisfy $\phi$. Thus, $M,w\models_K\square\phi$ in the Kripke model.
    \end{itemize} 
\end{itemize}

\end{proof}

Fagin and Halpern \cite{fagin1991uncertainty} show that for every probability structure, there is an equivalent structure in which belief functions have been defined, and the converse is not true in general, just under specific circumstances.
Having seen the relationship between Kripke frames and augmented neighbourhood frames, now we will study this relationship for probabilistic Kripke frames and a subclass of belief neighbourhood frames, namely additive ones. We will base our proof in the previous ones and the results of the previous section.

\begin{proposition}\label{KtoMp} Let $\mathcal{M}^P_K =\langle W,\mu,R\rangle$ be a probabilistic Kripke frame. There is an additive belief neighbourhood frame which is modally equivalent to $\mathcal{M}^P_K$.
    
\end{proposition}
\begin{proof}
 If we have  $\mathcal{M}^P_K$, we first obtain  $\cap N_1(w)$ for each $w\in W$ in an analogous way to the procedure described in proposition \ref{kmandanm} for Kripke and neighbourhood models. It will be the core of an augmented neighbourhood \ref{prop:N0-1aug}. 
 
    Now let us consider all singletons  $\{w_i\}\subseteq R(w)=\cap N_1 (w)$. They will be the elementary sets in  $N_{> 0}(w)^{\mathcal{NC}}$. We want to define an additive neighbourhood frame, and then we know by proposition  \ref{moon-zappa} that the following condition has to be met: $\sum_{\{w_i\}\in N_{>0}(w)^{\mathcal{NC}}}$ $b(w,\{w_i\})=1$. Thus, for all singletons $\{w_i\}$, we define a function  $b_k(w,\{w_i\})=\mu(w,w_i)$.  We also specify that $b(w,\varnothing)=0$, and thus $\varnothing\in N_{\geqslant 0}(w)$. 

    We have to define the value of $b$ for the rest of the sets in $N_{\geqslant 0}
    (w)$. Neighbourhoods are monotonic \ref{monot-n}, so for every $A\subseteq W$, we are going to consider its interior   $\texttt{int}_N(A)$ \ref{def:interior}. We define $b(w,A)=b(w,\texttt{int}_N(A))$ (see lemma \ref{b-X-aa}). If $\texttt{int}_N(A)$=0, then $b(w,A)=0$.

\end{proof}
\begin{proposition} \label{NtoPp}
Let $\mathcal{M}^B_N=\langle W,b,\{N_{\geqslant\alpha,\alpha\in [0,1]}\}\rangle$ be   an additive belief neighbourhood frame. There is a probabilistic Kripke frame which is modally equivalent to $\mathcal{M}^B_N$.

\end{proposition}
\begin{proof}    
    For all $w\in W$, we define a binary relationship $R_N:W\times W$ in this way: $wR_N w'$ iff $w'\in \cap N_1(w)$. We will denote by $R_N(w)$ the set of all $w'$ in this relation. 

    We know from proposition \ref{decomp-core} that $\cap N_1(w)$ is equal to the disjoint union of the elementary sets in $N_{> 0}(w)^{\mathcal{NC}}$. The frame is additive, so the sum of the beliefs of the elementary sets is equal to 1. Thus, we define a function $\mu_N:W\times W\to [0,1]$ such that $\mu_N(w,w_i)=0$ if  $w'\notin \cap N_1(w)$ and $\mu(w,w_i)>0$ otherwise. The value of $\mu_N(w,w_i)>0$ in this second case will be determined in the following step.
    
    Now we must consider each elementary set $E_i$. It will correspond to one or more $w_i$ in $R_N(w)$.  Since each set is disjoint, the exact value of $\mu(w,w_i)$  is arbitrary, as long as the sum of $\mu_N(w,w_i)$ for all $w_i$ within a given elementary set $E_i$ is equal to $b(w,E_i)$. We make this for all elementary sets.
    
    Since the frame is additive and $\cap N_1 (w)=\bigcup_{i=1}^n E_i$ for $\{E_i\}$ disjoint sets, then $b(w,\cap N_1 (w))=1=\sum_{i=1}^n b(w,E_i)$.
    Thus, $\sum_{w_i\in W}\mu_N(w,w_i)=1$.
\end{proof}

As a consequence, we have the following result:

\begin{proposition}\label{Meq-kn}
\textbf{(Equivalence between probabilistic and belief models.)} The class of probabilistic Kripke models is modally equivalent to the class of of additive  belief models.
\end{proposition}
\begin{proof}
Analogous to proposition \ref{kmandanm}. If we add a valuation $v:\texttt{At}\to W$ to the previous frames, we obtain the corresponding models. We just need to consider modal operators:

\begin{itemize}
\item \textbf{From Kripke to neighbourhood models:} If $\mathcal{M}^P_K,w\models B_{\geqslant \alpha}\phi$, then there is a subset of $R(w)$ such that the sum of their probabilities $\mu(w,w_i)$ is a certain value $\alpha'\geqslant\alpha$, and the analogue value for $\neg\phi$ is a certain value  $\beta$ such that $\alpha'+\beta=1$. By applying the process described in proposition \ref{KtoMp}, the elementary sets of the modally equivalent neighbourhood model are formed by a series of singletons, one for each world in $R(w)$. The union of all the elementary sets corresponding to worlds in $R(w)$ that satisfy $\phi$ is $\texttt{int}_N(\llbracket\phi\rrbracket_M)$.  $b(w,\llbracket\phi\rrbracket_M)=b(w,\texttt{int}_N(\llbracket\phi\rrbracket_M)$ is the sum of the beliefs of the elementary sets corresponding to worlds satisfying $\phi$. Thus $b(w,\llbracket\phi\rrbracket_M)=\alpha'\geqslant \alpha$. The same reasoning applies to $\neg\phi$, giving a value of $b(w,\llbracket\neg\phi\rrbracket_M)=\beta$. Therefore $\mathcal{M}^B_N\models \underline{B}_{\geqslant \alpha'\geqslant \alpha}\phi$ and $\mathcal{M}^B_N\models \underline{B}_{\geqslant \beta}\neg\phi$. We know that this logically entails $\mathcal{M}^B_N\models B_{\geqslant \alpha}\phi$.

\item \textbf{From additive neighbourhood to Kripke models:} Let us have $\mathcal{M}^B_N,w\models_N B_{\geqslant \alpha}\phi$. We first obtain $R(w)$ from $\cap N_1(w)$ as described in proposition \ref{NtoPp}: We define $R_N\colon W\times W$ such that $wRw_i$ if $w_i \in \cap N_1(w)$. A subset of the worlds in $\cap N_1(w)$ satisfies $\phi$, and this subset is also a subset of $\llbracket\phi\rrbracket_M$. More precisely, this subset, $\cap N_1(w)\cap \llbracket\phi\rrbracket_M=\texttt{int}_N (\llbracket\phi\rrbracket_M)$, and  by lemma \ref{b-X-1}, $b(w_1,\texttt{int}_N(\llbracket\phi\rrbracket_M))=b(w_1,\llbracket\phi\rrbracket_N)$. Thus, this value is at least $\alpha$, and therefore, the worlds in $R(w)$ that satisfy $\phi$ must sum up a probability of at least $\alpha$. Therefore, $\mathcal{M}^P_K,w\models_K B_{\geqslant \alpha}\phi$.
\end{itemize}

\end{proof}
Not all belief neighbourhood frames have a modally equivalent probabilistic neighbourhood frame: If the neighbourhood frame is not additive, we cannot obtain a modally equivalent Kripke frame. Instead, there will be an infinite family of of probabilistic Kripke frames that are compatible with that belief neighbourhood model. This means that we can obtain an infinite family of Kripke models, and the formulas that are satisfied on them logically entail the formulas of the neighbourhood model, but the converse is not true in general.

\begin{example}{(Example 4: From a probabilistic Kripke model to a belief neighbourhood model and back.)}\label{example4}
\textbf{From probabilistic Kripke model to belief neighbourhood model:}  The $\mathcal{M}^P_K$ in Example 1  (\ref{example1}) can have a modally equivalent belief neighbourhood model. Let us study just one world, the process in the other ones is identical. Let us take $w_1$, we recall that in $\mathcal{M}^P_K$, $\mu(w_1,w_1)=0$, $\mu(w_1,w_2)=0.4$, $\mu(w_1,w_3)=0.6$, $\mu(w_1,w_4)=0$. We will assign a certain $b$ to every subset of $R(w_1)=\{w_2,w_3\}$. $b(\varnothing)=0$, and thus $b(w_1,\{w_3\})=\mu(w_1,w_3)$=0.6, and $b(w_1,\{w_2\})=\mu(w_1,w_2)=0.4$. $b(\{w_1\}=0$. The belief function for the rest of the subsets of $W$ is defined from the belief function of their interior, that it is some of the previous sets. For example, $b(w_1,W)=b(w_1,\{w_2,w_3\})=1$, and $b(w_1,\{w_1,w_3\})=b(w_1,\{w_3\})=0.6$.

If we apply the valuation function that we have defined in Example 1, $\{w_1,w_3\}=\llbracket p\rrbracket_M$, and thus $b(w_1,p)=0.6$. Similarly, $\{w_2,w_4\}=\llbracket \neg p \rrbracket_M$, and thus $b(w_1,\neg p)=0.4$. That is, $b(w_1,p)+b(w_1,\neg p)=1$.

\textbf{From additive belief neighbourhood model to probabilistic Kripke model:} Now let us see the belief neighbourhood model we have obtained. Let us take $w_1$ again, for example, and let us observe that $\cap N_1 (w_1)=\{w_2,w_3\}$. Both $\{w_2\}$ and $\{w_3\}$ are elementary sets in $N_{\geqslant 0}(w_1)$, and we know that $b,(w_1,\{w_2\})+b,(w_1,\{w_3\})= 1$, and thus the frame is additive.  

Let us define $R_N\colon W\times W$ and $\mu_N\colon W\times W\to [0,1]$, and $\mu_N(w_1,w_2)=0.4$ and $\mu_N(w_1,w_3)=0.6$. Since each elementary set $E_i$ is formed by a single world, we do not have to distribute the value of $b(w_1,E_i)$ among several worlds. For several worlds, we should operate in this way: Let us assume that $E_1=\{w_2,w_5\}$, then $\mu_N(w_1,w_2)+\mu_N(w_1,w_5)=b(w_1,E_1)$, for example $\mu_N(w_1,w_2)=0.3$ and $\mu_N(w_1,w_5)=0.1$. Since $w_1$ and $w_4$ are not in $\cap N_1(w_1)$, then $\mu_N(w_1,w_1)=0$ and $\mu_N(w_1,w_4)=0$. We have that $R_N(w_1)=\{w_2,w_3\}$, which is the set of worlds in $\cap N_1(w_1)$. We verify that $\mu_N(w_1,w_2)+\mu_N(w_1,w_3)=1$, which is the expected value.

\end{example}

\begin{example}{(Example 5: Belief neighbourhood model with no modally equivalent probabilistic Kripke model.)}\label{example5}
We can take the belief neighbourhood model that we have created in the previous example and just make two small changes. Let us make $b(w_1,\{w_2\})=0.4$ and $b(w_1,\{w_3\})=0.3$, but $b(w_1,\{w_2,w_3\})=1$. We have that $b(w_1,\{w_2,w_3\})$ $> b(w_1,\{w_2\})+b(w_1,\{w_3\})$ and thus additivity is not held \ref{moon-zappa}. Thus, there is not a Kripke frame that is modally equivalent to this neighbourhood frame, but infinite Kripke frames, these with $\mu_N(w_1,w_2)\geqslant b(w_1,\{w_2\})=0.3$, $\mu_N(w_1,w_2)\geqslant b(w_1,\{w_3\}=0.4$,  and  $\mu_N(w_1,w_2)+\mu_N(w_1,w_3)=1$. 

Let us suppose a certain Kripke frame such that $\mu(w_1,w_2)=0.4$ and $\mu(w_1,w_3)=0.6$, being $R(w_1)=\{w_2,w_3\}$. We have that $\mathcal{M}^P_K,w_1\models_K B_{0.4}p$, for example. $B_{0.4}p$ logically entails $\underline{B}_{0.3}p$, which is satisfied in $w_1$ in the neighbourhood model. However, $\underline{B}_{0.3}p$ does not logically entail $B_{0.4}p$. 
\end{example}

\section{Concluding remarks and future work}\label{secfinal}

\textbf{Concluding remarks. } It has been shown that every probability can be understood as an additive belief function \ref{prandbel}. In a probability Kripke model, it is possible to express subjective probabilities, which also are additive belief functions \ref{pr-bel-KM}, but it is not possible to express arbitrary belief functions. In a neighborhood belief model, a probability is a special case of belief functions, additive ones. Thus, neighborhood models allow for a more expressive interpretation by distinguishing between probability and belief functions. In the final section, we have seen how to interpret this in terms of relationships between models: Every probabilistic Kripke model can be converted into a modally equivalent neighborhood model, but the reverse is only possible for additive belief neighbourhood models \ref{Meq-kn}. They form a proper subclass of of belief neighborhood models.

\textbf{Future works.}
In this work we have deliberately omitted the consideration of conditional probability and conditional belief, which arise as an immediate logical extension. In future works, an possible line of research is their study and their connection to basic logical operations, such as probabilistic deduction and the application of Bayes' theorem. Specifically, our aim is to integrate these operations within the framework of belief neighborhood models, considering these operations on probabilistic Kripke models as equivalent to their counterparts in a specific subclass of neighbourhood models.

\vspace{5mm}

\textbf{Acknowledgements.} Thanks to the two anonymous peer reviewers who greatly helped with their suggestions to improve the final version of this article.

\printbibliography 

\end{document}